\documentclass[12pt]{amsart}

\usepackage{epic,amsfonts,color,epsfig,amssymb,latexsym,amsmath,hyperref}
\usepackage[super]{nth}
\usepackage{color}

\newtheorem{thm}{Theorem}[section]

\newtheorem{cor}[thm]{Corollary}
\newtheorem{lem}[thm]{Lemma}

\newtheorem{ex}[thm]{Example}

\newtheorem{dfn}[thm]{Definition}

\newtheorem{rem}{Remark}[section]

\newcommand{\abs}[1]{\left\lvert#1\right\rvert}

\newcommand{\norm}[1]{\left\lVert#1\right\rVert}
\newcommand{\Int}{\int\limits}

\author{Alperen A. Erg\"{u}r}
\email{erguer@math.tu-berlin.de}
\address{Technische Universit\"at Berlin, Institut f\"ur Mathematik, Sekretariat MA 3-2, Straße des 17. Juni 136, 10623, Berlin, Germany }
\thanks{A.E.\ was partially supported by Einstein Foundation, Berlin.}
\title[Approximating Nonnegative Polynomials via Spectral Sparsification]{\mbox{} \\ 
\vspace{-1in} Approximating Nonnegative Polynomials via Spectral Sparsification}

\begin{document}

\maketitle

\vspace{-.4in}

\begin{abstract} 
We study polyhedral approximations to the cone of nonnegative polynomials. We show that any constant ratio polyhedral approximation to the cone of nonnegative degree $2d$ forms in $n$ variables has to have exponentially many facets in terms of $n$. We also show that for  fixed $m \geq 3$, all linear $m$ dimensional sections of the nonnegative cone that include $(x_1^2+x_2^2+\ldots + x_n^2)^d$ has a costant ratio polyhedral approximation with $O(n^{m-2})$ many facets. Our approach is convex geometric, and parts of the argument rely on the recent solution of Kadison-Singer problem. We also discuss a randomized polyhedral approximation which might be of independent interest.  
\end{abstract}

\begin{center}
  \emph{Dedicated  to Nuriye and Semih, and to thousands of souls longing for justice.}
\end{center}

\section{Introduction}

Let $P_{n,2d}$ be the vector space of  real homogenous degree $2d$ polynomials in $n$ variables. The elements of $P_{n,2d}$ that are nonnegative on the sphere form a full dimensional cone. Membership problem of this cone is algorithmically equivalent to global optimization of polynomials.  In the case of quadratics ($d=1$), membership of the cone can be checked efficiently. Starting with the case $d=2$, membership problem is NP-Hard \cite{ahmadi}.  

In general, it also seems hard to provide upper bounds for the complexity of nonnegative cone membership problem. However, a very interesting result of  Pebay, Rojas and Thompson shows that for any fixed $\delta > 0$, deciding if the supremum of a polynomial with $n$ variables and $n+n^{\delta}$ monomials exceeds a certain given number is $NP_{\mathbb{R}}$-complete  \cite{roj1}. Indeed, it is now clear that the complexity of membership problem for the cone of nonnegative polynomials is quite different in the case of sparse polynomials than the case of arbitrary degree $2d$ homogenous forms (dense polynomials). We do not intend to survey the literature on structured polynomial optimization here, we refer the reader to the work of De Wollf and Iliman, and the references therein \cite{timo1, timo2}.   

In this note, we are interested in polyhedral approximations to the cone of nonnegative polynomials. Discussion in the previous paragraph suggests that computational complexity of the approximation will be quite different in the dense and the sparse cases. Our results confirm this intuition as follows: We show that any constant ratio polyhedral approximation in the dense case has to have exponentially many facets. We also show that for any subspace $E \subset P_{n,2d}$ with $\dim(E)=m$ and $(x_1^2+x_2^2+\ldots+x_n^2)^d \in E$, there exists a polyhedral cone with $O(n^{m-2})$ many facets which provides a constant ratio approximation to the nonnegative elements of $E$.

We need to introduce some notation to resume. We denote the cone of nonnegative degree $2d$ forms in $n$ variables by $Pos_{n,2d}$. 

$$ Pos_{n,2d} : = \{ f \in P_{n,2d} : f(x) \geq 0 \; \text{for all} \; x \in S^{n-1}  \}$$

\noindent It is more convenient to work with a compact convex base of $Pos_{n,2d}$ instead of the unbounded cone itself. Note the following simple observation: for all $f \in Pos_{n,2d}$, we trivially have 

$$\Int_{S^{n-1}} f(x) \; \sigma(x) \geq 0 $$

\noindent where $\sigma$ is the uniform measure on $S^{n-1}$ with $\sigma(S^{n-1})=1$. This simple observation naturally suggests the following convex body $\widetilde{P}os_{n,2d}$ as a base for the nonnegative cone.

$$ \widetilde{P}os_{n,2d} := \{ f \in Pos_{n,2d} : \Int_{S^{n-1}} f(x) \; \sigma(x) = 1 \} $$

\noindent Approximating $\widetilde{P}os_{n,2d}$ with polytopes is equivalent to approximating $Pos_{n,2d}$ with polyhedral cones. Hence, in the rest of this note we are concerned with the polytope approximations to  $\widetilde{P}os_{n,2d}$. We begin with the familiar example of the cone of positive semidefinite matrices (PSD cone). 

\begin{ex}(The Cone of Nonnegative Quadratic Forms) For the case $d=1$, $Pos_{n,2}$ is the cone of nonnegative quadratics, or equivalently the PSD cone. We denote the trace of a matrix with $Tr$. Then, $ \widetilde{P}os_{n,2}$ can be expressed as follows.

$$ \widetilde{P}os_{n,2} := \{ Q \in Pos_{n,2} :  Tr(Q)=n \}$$

\noindent Now, let $c > 1$ be a constant and assume that $K$ is polytope with the following property.

$$ \widetilde{P}os_{n,2} - \mathbb{I}_n \subseteq K \subseteq c ( \widetilde{P}os_{n,2} -\mathbb{I}_n) $$ 

\noindent It follows from the work of Pokutta et al \cite{pokutta1, pokutta2} that $K$ has exponentially many facets.
\end{ex}

A spectrahedron is the intersection of the PSD cone with an affine linear space. Semidefinite programing methods optimize a linear objective function over a spectrahedron (see for instance Chapter 2 of \cite{semo}). The well established method of linear programing efficiently optimize a linear objective function over a polyhedron. The example above provides a comparison of the expressive power of semidefinite programing versus linear programing.  This was one of motivations for Pokutta and his collaborators in their work on the approximation limits of linear programing \cite{pokutta1,pokutta2}.  

Our first main theorem below provides an inapproximability result for arbitrary degree $d$. Our proof is direct and simple, and it is based on some basic Gaussian concentration inequalities.

\begin{thm} \label{intro1}
Let $c > 1$ be a constant, let $r=(x_1^2+x_2^2+\ldots+x_n^2)^{d}$, and suppose $K$ is a polytope with the following property.  

$$ \widetilde{P}os_{n,2d} - r \subseteq  K \subseteq c^{d} ( \widetilde{P}os_{n,2d} - r) $$

\noindent Then, $K$ has at least 

$$a_0e^{a_1\frac{n}{cd}}$$

\noindent  many facets where $a_0$ and $a_1$ are absolute constants. 
\end{thm} 
In Theorem \ref{intro1} we assume $d$ is any fixed degree, and one has the liberty to vary $c$ to be any real number greater than $1$. The constants $a_0$ and $a_1$ are some fixed numbers independent of $n$, $d$ and $c$. For instance, if one sets $c=5$, the conclusion of Theorem \ref{intro1} is that any polyhedral approximation to $\widetilde{P}os_{n,2d}$ with accuracy $5^d$ has to have $\Omega\left( \exp(\frac{n}{5d}) \right)$ many facets.

After the inapproximability result for the dense case, we consider polyhedral approximations for structured subspaces of polynomials. For a subspace $E \subset P_{n,2d}$, we denote the cone of nonnegative elements and its compact base as follows.

$$ Pos_{E} : = \{ f \in E : f(x) \geq 0 \; \text{for all} \; x \in S^{n-1}  \}$$

$$ \widetilde{P}os_{E} := \{ f \in Pos_E : \Int_{S^{n-1}} f(x) \; \sigma(x) = 1 \} $$

\noindent We will explain in the third section that for the $\widetilde{P}os_{E}$ definition to be meaningful we need to assume $r = (x_1^2+ x_2^2 + \ldots + x_n^2)^d  \in E$. From this point on and throughout the paper we always assume $r$ is included in the subspaces we consider.  We would like to present an example of such a subspace which is due to  Choi, Lam and Reznick  \cite{reznick}.

\begin{ex}(Even Symmetric Sextics)
Let $M_k(x)=\sum_{i=1}^n x_i^{k}$, and consider the following vector space $E$.
 $$E := span \{ M_6 , M_2M_4, M_2^3 \} $$

\noindent $E \subset P_{n,6}$, and it is the subspace formed by even symmetric forms. Observe that $E$ satisfies our assumption that $r \in E$ since $M_2^3=(x_1^2+x_2^2+\ldots+x_n^2)^3=r$. It follows from the results of Choi, Lam and Reznick  \cite{reznick} that $\widetilde{P}os_{E}$ is precisely a regular $n$-gon.
\end{ex}

In their beautiful paper \cite{reznick}, Choi, Lam and Reznick exploit algebraic properties of even symmetric sextics to conclude that the set of nonnegative elements is precisely the cone over a regular $n$-gon. It is not clear how to generalize their techniques to less structured families of sparse polynomials. 

We consider arbitrary subspaces of forms with the only assumption that the element $r$ is included the subspace. A corollary of our main theorem below is that for any fixed $m \geq 3$ and $n \rightarrow \infty$, all $m$ dimensional sections of $Pos_{n,2d}$ that include the element $r$ is roughly a polyhedral cone with $O(n^{m-2})$ many facets.  

\begin{thm} \label{intro2}
Let $E \subset P_{n,2d}$ be a linear subspace of fixed dimension $m$ where $3 \leq m \leq \frac{n}{e}$. Also assume that $r \in E$. Then, there exists a polytope $K$ with $O(n^{m-2})$ many facets which satisfies the following.

$$ \widetilde{P}os_{E} - r \subseteq  K \subseteq (1+ \frac{n}{m})^{\frac{3m}{n}} ( \widetilde{P}os_{E} - r) $$
 
\end{thm}

For the special case of quadratics, Theorem \ref{intro2} yields the following corollary.

\begin{cor}
Let $E$ be an $m$ dimensional affine linear space of real $n \times n$ matrices with $m \leq \frac{n}{e}$ and $\mathbb{I}_n \in E$. We denote the spectrahedron defined by the intersection of the PSD cone with $E$ by $Pos_E$. We define the following base $\widetilde{P}os_E$ for the spectrahedron $Pos_E$.

$$ \widetilde{P}os_E := \{ Q \in Pos_E :  Tr(Q)=n \} $$

\noindent Then, there exists a polytope $K$ with $O(n^{m-2})$ facets which satisfies the following inclusions.

$$ \widetilde{P}os_{E} - \mathbb{I}_n \subseteq  K \subseteq (1+\frac{n}{m})^{\frac{3m}{n}} ( \widetilde{P}os_{E} - \mathbb{I}_n ) $$
\end{cor}

Even though we stated Theorem \ref{intro2} as there exists a polytope $K$  satisfying the desired approximation, we actually prove existence of a polytope $K$ where all the facets are defined by single pointwise evaluations. In this respect, our construction relates to a basic question: Given an $\varepsilon > 0$ and $f \in E$ with $\Int_{S^{n-1}} f(x) \; \sigma(x) > 0$, how many pointwise evaluations are needed to certify that $  f(x) \geq - \varepsilon ( \Int_{S^{n-1}} f(x) \; \sigma(x) )$  for all $x \in S^{n-1}$ ?  

\noindent Theorem \ref{intro2} gives an estimate for $\varepsilon=(1+ \frac{n}{m})^{\frac{3m}{n}}-1$.

The proof of Theorem \ref{intro2} exploits convex geometric properties of $Pos_E$ through spectral sparsification. In particular, we use results of Friedland and Youssef \cite{youssef} which is based on the recent solution of Kadison-Singer problem. The solution of Kadison-Singer problem and results based on it (including Theorem \ref{intro2}) are not constructive. In that respect, we also study random construction of a polyhedral approximation. To state our random approximation result we need a little more terminology. For all $p \geq 1$, we define $p$-norm of a homogenous polynomial $f$ as follows.

$$ \norm{f}_p := ( \Int_{S^{n-1}}  \abs{f(v)}^{p}  \; \sigma(v) )^{\frac{1}{p}}$$

\noindent where $\sigma$ is the uniform measure on $S^{n-1}$.  We also define the following quantity for subspaces $E$ of $P_{n,2d}$.

$$ M(E) : = \max_{f \in E} \frac{\norm{f}_2^2}{\norm{f}_1^2} $$

Now we are ready to state the random polyhedral approximation theorem.

\begin{thm} \label{intro3}

Let $E \subset P_{n,2d}$ be a subspace with $r \in E$ and $\dim(E)=m$. Let $\mu$ be the measure defined in the last section of this paper. For a given $\alpha$ with $0 < \alpha < 1$, we set $t=O \left( m^2 (\frac{M(E)}{1-\alpha})^m \ln(\frac{M(E)}{1-\alpha}) \right)$. Let $x_1,x_2,\ldots,x_t$ be independent random vectors in $E$ distributed according to $\mu$. We define a polytope $K_{\alpha}$ as follows.

$$ K_{\alpha}= \{ y \in E : \langle y , r \rangle =0 \; \text{and} \; \langle y , x_i \rangle \geq -1 \; \text{for all} \; 1 \leq i \leq t \} $$

\noindent Then $K_{\alpha}$ satisfies the following inclusions with probability at least $1-4(\frac{1-\alpha}{M(E)})^{m^2}$.
$$ \widetilde{P}os_E - r \subseteq K_{\alpha} \subseteq \frac{1}{\alpha}(\widetilde{P}os_E - r) $$
\end{thm}
 
\noindent Bounding $M(E)$ from above seems to require more information than just the dimension of the subspace. However, it easy to prove $M(E) \leq 2^{2d}$ for any subspace $E$ of $P_{n,2d}$ (see corollary \ref{yep}). 

As a special case of Theorem \ref{intro3}, consider $2d=4$ and $\alpha=\frac{n-1}{n}$. Then, $K_{\frac{n-1}{n}}$ has  $O(m^2 (16n)^{m}\ln(n))$ many facets and it satisfies 

$$ \widetilde{P}os_E - r \subseteq K_{\frac{n-1}{n}} \subseteq \frac{n}{n-1}(\widetilde{P}os_E - r)$$

\noindent with probability greater than $1-4 n^{-m^2}$.

The rest of the paper is structured as follows: In the second section, we review the background material coming from geometric functional analysis. In the third section, we discuss convex geometric properties of the cone of nonnegative polynomials. In the fourth section, we prove Theorem \ref{intro1} using Gaussian concentration inequalities and convex geometric duality introduced in the third section. In the fifth section, we prove Theorem \ref{intro2} using the tools introduced in the second and the third sections. Finally, we prove Theorem \ref{intro3} in the last section using a tool coming from computational geometry, namely the epsilon-net theorem.   

\section{John's Theorem and Spectral Sparsification}

We begin with recalling a fundamental theorem in convex geometry due to Fritz John \cite{john}.

\begin{thm}(John's Theorem) \label{john}
Every convex body $K \subset \mathbb{R}^n$ is contained in a unique ellipsoid of the minimal volume $E_{min}$.  Moreover, 

$$ \frac{1}{n} E_{min}  \subset K \subset E_{min} .$$

\noindent The minimal volume ellipsoid $E_{min}$ is the Euclidean unit ball $B_2^{n}$ if and only if the following conditions are satisfied:  $K \subset B_2^n$, there are unit vectors $(u_i)_{i=1}^m$ on the boundary of $K$ and positive real numbers $c_i$ such that

$$ \sum_{i=1}^m c_i u_i=0$$

\noindent and for all $x \in \mathbb{R}^n$ we have
$$ \sum_i c_i \langle u_i , x \rangle^2 = \norm{x}_2^2$$

\end{thm}

If the minimal volume ellipsoid of a convex body $K$ is the unit ball, we say $K$ is in John's position. Henceforth, we call the conditions in Theorem \ref{john} characterizing the John's position as John's decomposition. One way to view John's decomposition is to observe that the family of unit vectors $\{ u_i \}_{i=1}^{m}$ work like an  orthogonal basis in $\mathbb{R}^{n}$. This phenomenon is studied in depth by frame theory. The lemma below can be found virtually in any frame theory textbook.

\begin{lem} \label{frame}
We denote the map that sends $x$ to $\langle  x , y \rangle z$ by $ y \otimes z$. Then the following are equivalent

\begin{enumerate}
\item   
$$ \mathbb{I}=\sum_{i} c_i u_i \otimes u_i $$
\item  For every $x \in \mathbb{R}^n$
$$ x=\sum_{i} c_i \langle x,u_i \rangle u_i $$
\item For every $x \in \mathbb{R}^n$ 
$$ \sum_i c_i \langle u_i , x \rangle^2 = \norm{x}_2^2$$

\end{enumerate}
\end{lem}

Another perspective on John's decomposition is to view the decomposition as a discrete measure supported on the vectors $u_i$ with weights $c_i$, and the identity being the covariance matrix of the measure. This measure theoretic interpretation is formalized in the notion of isotropic measures which we present below. 

\begin{dfn}
A finite Borel measure $Z$ on the sphere $S^{n-1}$ of a $n$ dimensional real vector space $V$ is  said to be isotropic if 

$$   \norm{x}_2^2 = \int_{S^{n-1}}  \langle x ,u \rangle^2 dZ(u)  $$

\noindent for all $x  \in V$. Moreover, we define the centroid of a measure $Z$ supported on the sphere $S^{n-1}$ as 

$$ \frac{1}{Z(S^{n-1})} \Int_{S^{n-1}} u \; Z(u)  .$$

\noindent We say the measure is centered at $0$ if the centroid is the origin.
\end{dfn}

An isotropic measure supported on the sphere with centroid $0$ is the continuous analog of John's decomposition. It is known that a convex body is in John's position if and only if the touching points of the convex body to the unit ball is the support of an isotropic measure with centroid $0$ \cite{petros}. The advantage of this continuous point of view is that interesting convex bodies such as convex hull and Minkowski sum of compact group orbits, or the dual of the cone of nonnegative polynomials are easily shown to support an isotropic measure with their (possibly) infinitely many touching points to the unit ball.

In general, convex bodies with fewer than $\frac{n(n+3)}{2}$ touching points to their minimal volume ellipsoid form a dense family in the space of convex bodies \cite{gruber}. The $O(n^2)$ many touching points in this dense family brings the problem of sparsification in John's decomposition. The main goal of this approach is to find $O(n)$ many vectors among the initial decomposition which form an approximate decomposition of identity. This line of reasoning is closely related to the recent solution of Kadison-Singer problem; we begin our discussion with a remarkable theorem of Batson, Spielman and Srivastava which was  a precursor to the solution of Kadison-Singer problem ( see Theorem 1.6 of \cite{siri1}).

\begin{thm} \label{ss}
Fix $\varepsilon \in (0,1)$ and $n,m \in \mathbb{N}$. For every $x_1,x_2,\ldots,x_m \in \mathbb{R}^n$ there exist $s_!,s_2,\ldots,s_m \in [0,\infty)$ such that

$$ \# \{ s_i :  s_i \neq 0 \} \leq \left\lceil  \frac{n}{\varepsilon^2} \right\rceil $$

\noindent and for all $y \in \mathbb{R}^n$ we have

$$ (1-\varepsilon)^2 \sum_{i=1}^{m} \langle x_i , y \rangle^2 \leq \sum_{i=1}^m  s_i \langle x_i , y \rangle^2 \leq (1+\varepsilon)^2 \sum_{i=1}^{m} \langle x_i , y \rangle^2 $$  

\noindent In particular,  for the case $\sum_{i=1}^m x_i \otimes x_i= \mathbb{I}_n$ we have

$$ (1-\varepsilon)^2  \mathbb{I}  \preceq \sum_{i=1}^{m} s_i x_i \otimes x_i  \preceq (1+\varepsilon)^2 \mathbb{I}  $$
\end{thm}

\noindent Theorem \ref{ss} was recently refined by an article of Friedland and Youssef \cite{youssef}. Friedland and Youssef's work uses the solution of Kadison-Singer problem as an intermediate step and then provides refined estimates on a suite of problems including spectral sparsification, restricted invertibility and isomorphic Dvoretzky problem. The following result ( Theorem 4.1 of \cite{youssef}) will be used in the fifth section. 

\begin{thm} \label{pierre}
There exists a universal constant $c$ such that the following holds. Let $\varepsilon >0$ and $\{ c_i , x_i \}_{i=1}^m$ be a John's decomposition of identity in $\mathbb{R}^n$ (i.e., $x_i$ are unit vectors and $c_i > 0$). Then there exists a multiset  $\sigma$ of indices from $[m]$ of size at most $\frac{n}{c\varepsilon^2}$ so that

$$  (1-\varepsilon) \mathbb{I}  \preceq \frac{n}{\abs{\sigma}} \sum_{i \in \sigma} (x_i-u)  \otimes (x_i-u) \preceq (1+\varepsilon) \mathbb{I}  $$

\noindent where $u=\frac{1}{\abs{\sigma}} \sum_{i \in \sigma} x_i$ satisfies $\norm{u} \leq \frac{2\varepsilon}{3\sqrt{n}}$. 
\end{thm}

\section{Convex Geometry of Nonnegative Polynomials} \label{cgpol}

In this section,  we would like to introduce modern convex geometry point of view on nonnegative polynomials. Our plan is to first introduce the concepts and results in the dense case, and then write the implications for the sparse case in a separate section. Most of the results in this section have appeared  in the literature \cite{barvinok3, blekherman} with possibly different proofs. 

We start with defining  an inner product on $P_{n,2d}$. For $f,g \in P_{n,2d}$, the inner product $\langle f , g \rangle$ is defined as follows.

$$  \langle f ,g \rangle = \Int_{S^{n-1}}   f(x) g(x) \; \sigma(x)   $$

\noindent where $\sigma$ is the uniform measure on the sphere $S^{n-1}$. It must be clear that the norm introduced by this inner product is the $2$-norm defined in the introduction. Throughout the paper this norm will be denoted by $\norm{.}_2$.

We consider the action of $SO(n)$ on $P_{n,2d}$. For $T \in SO(n)$ and $f \in P_{n,2d}$,  we denote the result of the action of $T$ on $f$ by $T \circ f$, and the action is defined by pointwise evalutions on $x \in S^{n-1}$ as follows.

$$ T \circ f (x) := f (T^{-1}x) $$ 

\noindent Since knowing all pointwise evaluations on the sphere uniquely defines the homogenous polynomial $T \circ f$, this action is well defined. For any $f,g \in P_{n,2d}$, we clearly have 

$$ \langle T \circ f , T \circ g \rangle = \langle f ,g \rangle .$$

\noindent Hence, the inner product is $SO(n)$ invariant. Now, we consider pointwise evaluation maps on the vector space $P_{n,2d}$. Let $v \in S^{n-1}$, and consider the following map.

$$ l_v : P_{n,2d} \rightarrow \mathbb{R}  \; \; , \; \; l_v(f)=f(v) $$

\noindent The operator norm of $l_v$ is defined as follows.

$$ \norm{l_v}=\max_{\norm{f}=1} \abs{l_v(f)}= \max_{\norm{f}=1} \abs{f(v)} $$ 

\noindent For any two arbitrary unit vectors $u,v \in S^{n-1}$, one can find $T \in SO(n)$ such that $T(u)=v$. Then for all $f \in P_{n,2d}$, one has $T \circ f(v)=f(u)$. Since, $\norm{.}_2$ is invariant under the $SO(n)$ action, we immediately have $\norm{l_v}=\norm{l_u}$. Hence, the operator norm $\norm{l_v}$ is fixed for all $v \in S^{n-1}$. 

We introduced a Hilbert space structure on $P_{n,2d}$, so we have the Riesz Representation Theorem. That is, for all $v \in S^{n-1}$ there exists a corresponding unique $p_v \in P_{n,2d}$ such that  for all $f \in P_{n,2d}$, we have 

$$ l_v(f) = f(v) = \langle f , p_v \rangle .$$

\noindent Since $\norm{l_v}=\norm{p_v}_2$, and since the norm of $l_v$ is fixed over the sphere, we conclude that $\norm{p_v}_2$ is fixed over the sphere as well.  

Results of this sections are basis independent; the polynomials $p_v$ only depend on the inner product. In the lemma below, we write a concrete expansion of $p_v$ for an arbitrary orthonormal basis and derive some basic properties.

\begin{lem} \label{zonal}
\noindent Let $u_1, u_2, \ldots, u_N \in P_{n,2d}$ be  an orthonormal basis for $P_{n,2d}$ where $N=\binom{n+d-1}{d}$. For every $v \in S^{n-1}$, we define the following polynomial $p_v$.

$$ p_v(x) := \sum_{i=1}^N u_i(v) u_i(x)  $$
 
\noindent Then, $p_v$ possess the following properties:

\begin{enumerate}
\item For all $q \in P_{n,2d}$, we have

$$ \langle q , p_v \rangle = q(v) .$$

\item  For $v, w \in S^{n-1}$  and $T \in SO(n) $ we have  the following equality.

$$  p_v(w)=p_w(v)= p_{T(v)}(T(w))  $$

\item  For all $v \in S^{n-1}$ the following holds. 

$$ N=p_v(v)=\norm{p_v}^2  $$

\item The following holds for all $q \in P_{n,2d}$.

$$ \frac{max_{v \in S^{n-1}} \abs{q(v)}   }{\norm{q}_2}  \leq \sqrt{N} $$

\end{enumerate}
\end{lem}

\begin{proof}
Given  $q \in P_{n,2d}$ we have 

$$ q(v)= \sum_{i=1}^N \langle q ,u_i \rangle u_i(v)= \langle q , \sum_{i} u_i(v) u_i \rangle= \langle q ,p_v \rangle  $$

For any $T \in SO(n) $ and any $ f \in P_{n,2d}$, we have

$$  \langle f , p_{T(w)} \rangle = f(T(w)) = \langle T^{-1} \circ f , p_w  \rangle = \langle f , T \circ p_w \rangle  $$

\noindent Since $f$ is arbitrary, this proves  $p_{T(w)}= T \circ p_w $ and it completes the proof of second claim.  

Now we would like to show that $p_v(v)=N$ for all $v \in S^{n-1}$. Since $p_v(v)=p_{w}(w)$ for all $v,w \in S^{n-1}$, we have

$$  p_v(v)=\Int_{S^{n-1}} p_v(v) \; \sigma(v) =  \Int_{S^{n-1}}  \langle p_v , p_v \rangle \; \sigma(v) $$

\noindent Expanding the right most equation, we have

$$ p_v(v)= \Int_{S^{n-1}}  \langle \sum_{i} u_i(v) u_i ,  \sum_{i} u_i(v) u_i \rangle \; \sigma(v) = \sum_{i} \Int_{S^{n-1}} u_i(v)^2 \; \sigma(v)=N $$  

\noindent Last claim in the theorem statement is a direct application of the Cauchy-Schwartz inequality. 
\end{proof}

\noindent For any nonnegative polynomial $p \in Pos_{n,2d}$ we have

$$  \Int_{S^{n-1}} p(x) \; \sigma(x)   = \langle p , r \rangle \geq 0  $$

\noindent where $r=(x_1^2+x_2^2+\ldots + x_n^2)^{d}$. We observe that $\widetilde{P}os_{n,2d}= \{ p \in Pos_{n,2d} : \langle p ,r \rangle =1 \}$.   We denote by  $U$ the subspace of $P_{n,2d}$ consisting of polynomials orthogonal to $r$.

$$  U := \{ f \in P_{n,2d} : \langle f ,r \rangle = 0 \} $$

\noindent In other words, $U$ is the subspace of mean zero polynomials.

$$ U := \{  f \in P_{n,2d} :  \Int_{S^{n-1}}  f(x) \; \sigma(x)= 0  \} $$

\noindent Now we define a map from $S^{n-1}$ to $P_{n,2d}$ as follows.

$$ \phi:  S^{n-1} \rightarrow P_{n,2d} $$
$$  \phi(v)= p_v-r   $$

\noindent Observe that for all $v \in S^{n-1}$ we have $\norm{\phi(v)}_2=\sqrt{N-1}$. Moreover, we have   

$$  \langle \phi(v), r \rangle = \langle p_v -r ,r \rangle =  0 .$$

\noindent Hence $ \phi(S^{n-1}) \subset U $.  Now let $\sigma$ be the uniform measure on $S^{n-1}$ and let $\mu$ be the pushforward measure of $\sigma$ under $\phi$. For all $q \in U$ we have the following equality.

$$ \norm{q}^2= \Int_{S^{n-1}}  q(v)^2 \; \sigma(v) = \Int_{S^{n-1} } \langle q , p_v \rangle^2 \; \sigma(v) = \Int_{S^{n-1}}   \langle q , \phi(v) \rangle^2  \; \sigma(v)   $$

\noindent By definition of the pushforward measure we have

$$ \norm{q}^2=  \Int_{S^{n-1}}   \langle q , \phi(v)  \rangle^2 \; \sigma(v)  = \Int_{U} \langle q , u \rangle^2  \; \mu(u)  .$$

\noindent Therefore we observe that $\mu$ is an isotropic measure supported on the $\sqrt{N-1}$ scaled sphere of $U$. Hence $\mu$ creates a decomposition of identity! \\

\noindent In order to view the support of $\mu$ as a John's decomposition we also need to compute it's centroid. 

$$   q= \Int u \; \mu(u) = \Int_{S^{n-1}} (p_v -r) \; \sigma(v)   $$

\noindent  By construction, $q$ is invariant under the action of $SO(n)$. Therefore $q=a (x_1^2+x_2^2+\ldots + x_n^2)^{d}$ for some $a \in \mathbb{R}$. Since $q \in U$, we deduce that $a=0$. Hence the measure $\mu$ is centered at the origin. \\

\noindent Now we define the body of pointwise evalutaions;

$$ B := conv \left( Im (\phi) \right)  $$ 

\noindent  We observed that $\frac{1}{\sqrt{N-1}}B$ is convex hull of an isotropic measure supported on the sphere with centroid 0. It immediately follows from the discussion in the previous section that $\frac{1}{\sqrt{N-1}}B$ is in John's position. 

\noindent Now we consider the dual convex body $B^{\circ}$.

$$ B^{\circ} = \{ q \in U : \langle q , p \rangle \leq 1 \; \text{for all} \; p \in B  \} = \{  q \in U : \langle q , p_v-r \rangle  \leq 1  \; \text{for all} \; v \in S^{n-1} \} $$

\noindent By definition of $U$ and $p_v$,  we have $\langle q , p_v-r \rangle = \langle q , p_v \rangle=q(v)$ which shows the following equivalence.

$$ - B^{\circ} + r  = \widetilde{P}os_{n,2d}  $$

\noindent This nice convex geometric duality allows us to approximate $\widetilde{P}os_{n,2d}$ by approximating $B$.  

\begin{rem}
Readers who incline more toward algebraic geometry can think of $B$ as the convex hull of the $d$-th Veronese embedding. 
\end{rem}

\subsection{Structured Polynomials}

Let $E \subset P_{n,2d}$ be a linear subspace with $\dim(E)=m$ and $r=(x_1^2+x_2^2+\ldots + x_n^2)^d \in E$. Recall the definition of $Pos_E$.

$$ Pos_E := \{ f \in E : f(x) \geq 0 \; \text{for all} \; x \in S^{n-1} \}$$

\noindent We use the inner product induced by $P_{n,2d}$ on $E$. For all $f \in Pos_E$, we trivially have $\langle f , r \rangle \geq 0$. Now we recall the definition of the compact base $\widetilde{P}os_E$.

$$ \widetilde{P}os_E := \{ f \in Pos_E : \langle f ,r \rangle = 1 \} $$

\noindent Let $\Pi_E$ denote the orthogonal projection map on $E$. Using the notation introduced in the previous subsection, we have the following for all $f \in E$.

$$ f(v) = \langle f , p_v \rangle =  \langle f , \Pi_E(p_v) \rangle$$

\noindent As an example, since $r \in E$ we have $ \langle r , p_v \rangle = 1 = \langle r , \Pi_E(p_v) \rangle$.  Now we define a map $\phi_E$ as follows. 

$$ \phi_E:  S^{n-1} \rightarrow E $$
$$  \phi_E(v)= \Pi_E(p_v-r)=\Pi_E(p_v) -r   $$

\noindent First observation is that $ \langle \phi_E(v) , r \rangle = 0$. We define $U(E)$ to be the following subspace.

$$ U(E) := \{ f \in E : \langle f , r \rangle = 0 \} $$

\noindent  Hence, $\dim(U(E))= m-1$, and we have $\phi_E(S^{n-1}) \subset U(E)$. 

Second observation is that $\norm{\phi_E(v)}_2=(\norm{\Pi_E(p_v)}_2^2 - 1)^{\frac{1}{2}}$. This observation shows that $\norm{\phi_E(v)}_2$ can change at every point $v \in S^{n-1}$ in contrast to the situation in previous section. This change in the norm can occur because $E$ is not necessarily closed under the $SO(n)$ action, and certain directions on the sphere are preferred over others by the structure of the subspace $E$. 

We would like to continue with the isotropic measure observation of the previous section.  We define $\mu_E$ to be the pushforward measure of $\sigma$ (the uniform measure on $S^{n-1}$) to $E$ under the map $\phi_E$. Now, for all $q \in E$ we have

$$ \norm{q}^2=  \Int_{S^{n-1}}   \langle q , \phi_E(v)  \rangle^2 \; \sigma(v)  = \Int_{E} \langle q , u \rangle^2  \; \mu_E(u)  . $$

\noindent Therefore, $\mu_E$ is an isotropic measure. Proving that $\mu_E$ has centroid at $0$ is also easy. 

Even though we lost the control on the norms of $\norm{\phi_E(v)}_2$, $\mu_E$ being an isotropic measure has the following immediate consequence (which can be seen by taking the trace of the covariance matrix of $\mu_E$). 

$$ \Int_{S^{n-1}}  \norm{\phi_E(v)}_2 \; \sigma(v) = m-1 $$

\noindent Now, we define the body of pointwise evaluations in this setting as follows.

$$ B(E) := conv \{ \phi_E(v) :  v \in S^{n-1} \}  $$

\noindent For any $f \in E$ with $\langle f , r \rangle = 0$ and for all $v \in S^{n-1}$ we have the following relation.

$$ f(v) \geq -1  \Leftrightarrow \langle f , p_v \rangle \geq -1  \Leftrightarrow \langle f , \phi_E(v) \rangle \geq - 1 $$

\noindent Also note that $q\in \widetilde{P}os_E$ if and only if $\langle q-r , r \rangle = 0$ and $(q-r)(v) \geq -1$ for all $v \in S^{n-1}$. Hence,  we conclude

$$  \widetilde{P}os_E = - B(E)^{\circ} + r $$

\begin{rem} \label{rudelson}
One of the reasons that make isotropic measures appealing is a theorem of Rudelson. Let $x \in \mathbb{R}^m$ be an isotropic random vector. Let $x_1, x_2, \ldots, x_M$ be independent copies of $x$. Then, Rudelson's remarkably general theorem \cite{rudelson1} states the following.

 $$ \mathbb{E} \norm{  \frac{1}{M} \sum_{i=1}^M x_i \otimes x_i - \mathbb{I}  } \leq C \sqrt{ \frac{\log(m)}{M} } (\mathbb{E} \norm{x_i}_2^{\log M})^{\frac{1}{\log M}} $$

\noindent Hence, if one has any control on the $\max_{v \in S^{n-1}} \norm{\phi_E(v)}_2$, Rudelson's theorem provides a randomized way to obtain an approximate decomposition of identity, and that is all needed for the construction of a polyhedral approximation to $B(E)$.
\end{rem}

\begin{rem}
It turns out that the following can be proved without too much effort. For a fixed $m > 8n$, let $E$ be a random $m$ dimensional linear subspace of $P_{n,2d}$ drawn from the Haar measure on $Gr(\binom{n+d-1}{d},m)$. Then, the following hold for all $v \in S^{n-1}$

$$ \abs{ \norm{\phi_E(v)}_2 - \sqrt{m} } \leq \frac{\sqrt{m}}{2} $$

\noindent with probability greater than $1-\exp(-\frac{m}{8c_1n})$ where $c_1$ is an absolute constant.

This shows that subspaces of dimension $\Omega(n)$ typically have well controlled behavior in terms of the change in the norms of $\norm{\phi_E(v)}_2$. In this note, we are interested in the case where $m$ is a small fixed number independent of $n$. So, we skip the proof of this claim and leave it to reader who enjoys working with random projections. 
\end{rem}

\section{Limits of Approximation with Few Facets}

We have established the convex geometric duality between $ \widetilde{P}os_{n,2d}$ (the section of the cone of nonnegative polynomials) and $B$ (the convex body of pointwise evaluations). Thanks to this duality, searching for a polytope $Q$ with few facets that is sandwiched between  $  \frac{1}{c} \widetilde{P}os_{n,2d} $  and $\widetilde{P}os_{n,2d}$ (for some constant $c > 1$) is equivalent to searching for a polytope $P$ with few vertices that is sandwiched between $B$ and $cB$. In this section, we show that for any constant $c>1$, a polytope $P$ satisfying 

$$ B   \subseteq  P \subset c^d B$$

\noindent  has to have exponentially many vertices in terms of $n$.

Our result in this section is a direct application of basic properties of the Gaussian measure. Similar inapproximability results for the special case of quadratics  were obtained by Pokutta et al with a completely different approach \cite{pokutta1,pokutta2}.  

We start with presenting  two facts about the Gaussian measure that are going to be used in our proof.  First fact is a tail bound for polynomial maps on Gaussian random variables. 

\begin{lem} \label{1}
Let $f$ be a homogenous degree $2d$ polynomial with $\Int_{S^{n-1}} f(x) \; \sigma(x) = 0$, and let $\gamma_n \sim \mathcal{N}(0,Id)$ be the standard Gaussian measure on $\mathbb{R}^n$. Then, for all $t \geq \sqrt{n+2d}$ we have

$$ \gamma_n \left( \{ x : \abs{f(x)} \geq t^{2d} \norm{f}_2 \} \right) \leq a_0 \exp(-a_1 \frac{t^2}{n+2d}) $$

\noindent where $a_0$ and $a_1$ are positive absolute constants.
\end{lem}

\begin{proof}
We start by presenting a standard tail estimate for polynomials with normal random variables (see for instance Cor 5.5.7 in \cite{bogachev}). We denote the standard Gaussian measure on $\mathbb{R}^n$ with $\gamma_n$. Then, for a polynomial $f$ with $\Int_{S^{n-1}} f(x) \; \sigma(x) = 0$ the standard tail bound reads as follows.

$$  \gamma_n \left( \{ \abs{f(x)} \geq  s^{2d}  (\Int_{\mathbb{R}^n} f(x)^2 \; \gamma_{n}(x))^{\frac{1}{2}} \} \right) \leq a_0 e^{-a_1 s^2}  $$

\noindent where $a_0,a_1$ are absolute constants. Now we just need to rewrite this estimate with the norm $\norm{.}_2$ of this paper. Let us recall a basic integral identity:

$$ \Int_{S^{n-1}} f(x)^2 \; \sigma(x) = \frac{\Gamma(\frac{n}{2})}{2^{2d} \Gamma(\frac{n}{2}+2d)} \Int_{\mathbb{R}^n} f(x)^2 \; \gamma_{n}(x) $$

\noindent Therefore,  we have 

$$  (\Int_{\mathbb{R}^n} f(x)^2 \; \gamma_{n}(x))^{\frac{1}{2}} = 2^d \norm{f}_2 \left( \frac{\Gamma(\frac{n}{2}+2d)}{\Gamma(\frac{n}{2})} \right)^{\frac{1}{2}} \leq  \norm{f}_2  (n + 2d)^{d}.$$

\noindent Hence, for all $s \geq 1$ we have

$$ \gamma_n \left( \{ \abs{f(x)} \geq s^{2d} (n+ 2d)^d \norm{f}_2 \} \right) \leq  \gamma_n \left( \{ \abs{f(x)} \geq  s^{2d}  (\Int_{\mathbb{R}^n} f(x)^2 \; \gamma_{n}(x))^{\frac{1}{2}} \} \right) \leq a_0 e^{-a_1 s^2}.$$

\noindent Setting $t^2=s^2 (n+2d)$ completes the proof.

\end{proof}

The second fact we are going to use is a standard tail bound for the norm of a Gaussian vector (see for instance \cite{barvinok2}). 

\begin{lem} \label{2}
Let $v$ be a random vector distributed according to standard Gaussian measure $\gamma_n \sim \mathcal{N}(0,1)$ on $\mathbb{R}^n$. Then,
$$ \gamma_n \{ \norm{v}_2^2 \leq (1-\varepsilon)n \} \leq  e^{\frac{-\varepsilon^2n}{4}} .$$
\end{lem}

Now we are ready to present the main result of this section. 

\begin{thm} \label{inapprox}
Let  $B$ be the body of pointwise evaluations as defined in the previous section, let $c > 1$ be a constant, and assume there exists a polytope $P= conv \{ f_i :  1 \leq  i \leq N \}$ with the following property. 

$$  B  \subseteq P \subset c^d B  $$

\noindent Then, we have

$$\abs{N} \geq a_o e^{a_1\frac{n}{cd}}$$

\noindent where $a_0$ and $a_1$ are absolute constants.
\end{thm}

\begin{proof}
We use an idea of Barvinok, which is to study the maxima of a linear functional on the polytope to bound the number of its vertices \cite{barvinok1}. Our linear functional will be the pointwise evaluation map $l_v$ at a random Gaussian vector $v \in \mathbb{R}^n$. We aim to use the known properties of the convex body $B$ to arrive to a lower bound on the number of vertices of $P$.  We define a symmetric polytope $\widetilde{P}$ for convenience.

$$\widetilde{P} := conv \{ f_i, -f_i :  1 \leq i \leq N \}$$  

\noindent We observe that

$$   B  \subseteq \widetilde{P}  \subseteq  c^d \binom{n+2d-1}{2d}^{\frac{1}{2}} B_2^{m-1}   $$

\noindent Note that, here the unit ball is defined with respect to the $\norm{.}_2$ norm introduced by the inner product. 

Using the fact that maximum of a linear functional is attained at the vertices of a polytope and performing a basic union bound, we have the following inequality.

$$ \mathbb{P} \{ v:  \max_{f \in  \widetilde{P}  } f(v)  \geq \tau  \}  =  \mathbb{P} \{ v:  \max_{ f \in P} \abs{f(v)}  \geq \tau \} \leq \abs{N} \max_{1 \leq i \leq N} \mathbb{P} \{  v: \abs{f_i(v)} \geq \tau \}  $$

\noindent For any vector $v \in \mathbb{R}^n$, we define $\tilde{v}=\frac{v}{\norm{v}}$. Now using the properties of $p_{\tilde{v}}$ from Lemma \ref{zonal}, we have the following estimate from below. 

$$ \max_{f \in B}  f(v) \geq p_v(v)=\norm{v}^{2d}  p_{\tilde{v}}(\tilde{v}) \geq \binom{n+2d-1}{2d} \norm{v}^{2d}   $$ 

Therefore, we have the following lower bound.

$$  \mathbb{P} \{ v:  \max_{ f \in P} \abs{f(v)}  \geq \tau \}  \geq \mathbb{P} \{ v :  \norm{v}^{2d}  \geq \tau \binom{n+2d-1}{2d}^{-1} \} = \mathbb{P} \{ \norm{v}_2^2 \geq \tau^{\frac{1}{d}} \binom{n+2d-1}{2d}^{-\frac{1}{d}}  \}  $$

\noindent For all $f_i \in P$, we have $\norm{f}_2 \leq  c^d \binom{n+2d-1}{2d}^{\frac{1}{2}}$. Then, Lemma \ref{1} gives the following estimate.

$$ \mathbb{P} \{  v: \abs{f_i(v)} \geq t^{2d} \binom{n+2d-1}{2d}^{\frac{1}{2}}  \} \leq a_0 \exp(-a_1 \frac{t^2}{c(n+2d)})  $$ 

\noindent We set $\tau=t^{2d} \binom{n+2d-1}{2d}^{\frac{1}{2}}$. Then, the very first inequality in this proof gives us the following lower bound on $\abs{N}$. 

$$ \abs{N} \geq \frac{ \mathbb{P} \{ \norm{v}_2^2 \geq t^2 \binom{n+2d-1}{2d}^{-\frac{1}{2d}}  \}}{ \max_{1 \leq i \leq N} \mathbb{P} \{  v: \abs{f_i(v)} \geq t^{2d} \binom{n+2d-1}{2d}^{\frac{1}{2}}  \}}$$

\noindent We choose $t^2=\frac{n(n+2d)}{4ed}$, which ensures 

$$t^2 \binom{n+2d-1}{2d}^{-\frac{1}{2d}} \leq \frac{n}{2} $$

 and also 

$$\frac{t^2}{c(n+2d)} = \frac{n}{4ced}$$ 

\noindent  We obtained an inequality for bounding $\mathbb{P} \{  v: \abs{f_i(v)} \geq t^{2d} \binom{n+2d-1}{2d}^{\frac{1}{2}}  \}$ from above. We also have Lemma \ref{2} to bound $\mathbb{P} \{ \norm{v}_2^2 \geq t^2 \binom{n+2d-1}{2d}^{-\frac{1}{2d}}$ from below. Putting these two inequalities together we obtain the following estimate with $b_0=a_0^{-1}$ and $b_1=\frac{a_1}{4e}$. 

$$ \abs{N} \geq b_0 e^{b_1 \frac{n}{c d}} (1- e^{-\frac{n}{16}})   $$

\end{proof}

\section{Approximation Polytopes}
  
We  start this section with a simple observation: Suppose  that a convex body $K$	and a polytope $P$ satisfy $ P \subset K \subset \alpha P$. Then, for any invertible linear map $T$ we have the following inclusions.
   
   $$ T(P) \subset T(K) \subset \alpha T(P)$$
	
\noindent Therefore, to approximate $K$ with polytopes, we can select a suitable linear map $T$ and approximate the ``easier'' convex body $TK$ instead. 

In this section, we will prove existence of approximation polytopes to $B(E)$ with few facets. Based on the simple observation above, we assume without loss of generality that $B(E)$ is in John's position. Recall that by construction $B(E)$ lies inside the subspace $U(E) := \{ f \in E : \langle f , r \rangle = 0\}$. Therefore, being in John's position means that we have

$$  \frac{1}{m-1} B_2^{m-1} \subseteq  B(E)  \subseteq B_2^{m-1} $$

\noindent and the touching points of $B(E)$ to the unit ball of $U(E)$ form a John's decomposition of identity.

The following theorem is a direct corollary of the work of Friedland and Youssef (see Theorem \ref{pierre}).

\begin{thm} \label{cufcuf}
There exist a universal constant $c$ so that the following holds. For every $\varepsilon > 0$, there exists a multiset of vectors $S \subseteq S^{n-1}$ with $\abs{S} \leq \frac{m-1}{c\varepsilon^2}$ and corresponding polynomials $q_v \in \partial B(E) \cap B_2^{m-1}$ for all $v \in S$, so that 

$$   (1-\varepsilon) \mathbb{I} \preceq \frac{m-1}{\abs{S}} \sum_{ v \in S} q_v \otimes q_v  \preceq (1+ \varepsilon) \mathbb{I} $$

\noindent and 

$$\norm{ \frac{1}{\abs{S}} \sum_{v \in S} q_v } \leq \frac{2\varepsilon}{3\sqrt{m}} $$ 

\noindent where   $\mathbb{I}$ is the identity map on the subspace $U(E)$ of $E$.
\end{thm}

\begin{proof}
If $T$ is the map which puts $B(E)$ in John's position, then one can consider $T ( \phi_E(v))$ as the parametrization of $T B(E) \subset \mathbb{R}^{m-1}$ and apply Theorem \ref{pierre}  to the John's decomposition given by the touching points of $ T B(E)$ to the unit ball. 
\end{proof}

\begin{rem}
It is important to notice that the existence of a John's Decomposition of identity with $\frac{(m-1)(m+2)}{2}$ many vectors is already guaranteed by  Fritz John's classical article \cite{john}.   The goal of the spectral sparsification is to reduce the order  of number of  the vectors in the decomposition from quadratic to linear. Hence, all the results in the spectral sparsification literature are effective only for $\frac{1}{\sqrt{m}} \leq \varepsilon \leq 1$. 
\end{rem}

\begin{rem}
Note that $S$ being a  set or a multiset where some elements are counted with multiplicity does not affect the rest of our arguments. Therefore, from this point on we consider $S$ as a set for simplicity.  
\end{rem}

 Using Theorem \ref{cufcuf}, we will first construct a polytope $Q$ with $O(m)$ facets which gives the following rough approximation.

$$  \widetilde{P}os_E - r \subseteq Q \subseteq  m^{\frac{3}{2}}  (\widetilde{P}os_E - r) $$

Then, in the next subsection we will improve the accuracy of approximation by using a ``tensorization'' trick.  We start with a set of unit vectors $S$  with $\abs{S} \leq \frac{m-1}{c \varepsilon^2}$ which satisfies the following. 

$$  (1-\frac{\varepsilon}{4}) \mathbb{I} \prec T=\frac{1}{\abs{S}} \sum_{v \in S} f_v \otimes f_v \prec (1+\frac{\varepsilon}{4}) \mathbb{I} $$

\noindent Existence of such a set $S$ is guaranteed by taking $f_v=\sqrt{m-1}q_v$ in Theorem \ref{cufcuf}.  \\ 

Since $\norm{T- \mathbb{I}} \leq \frac{\varepsilon}{4}$, for any $p \in E$ we have the following upper and lower bounds.  

$$ (1-\varepsilon) \norm{p}^2 \leq \langle T p , p \rangle \leq (1+\varepsilon) \norm{p}^2  $$

\noindent Since $\langle T p , p \rangle = \frac{1}{\abs{S}} \sum_{v \in S} \langle p , f_v \rangle^2$, pigeon hole principle implies that there exists a $v \in S$ such that $ \abs{\langle f_v , p \rangle} \geq (1-\varepsilon) \norm{p} $. Now, we define our approximation polytope $P$ as follows.

$$ P := \{ f_v , -f_v : v \in S \}$$

\noindent By the above observation, for all $f \in E$ we have $ \max_{q \in P} \langle f , q \rangle \geq (1-\varepsilon) \norm{f}$. This implies  $(1-\varepsilon) B_2^{m-1} \subseteq P$. Since $f_v \in \sqrt{m-1}B(E)$ and $B(E)$ is in John's position $\norm{f_v} \leq \sqrt{m-1}$. In summary, we have

$$ (1-\varepsilon) B_2^{m-1} \subseteq P \subseteq \sqrt{m-1} B_2^{m-1} $$

This implies the following inclusion by John's Theorem.

$$ (1-\varepsilon) B(E) \subseteq P \subseteq m^{\frac{3}{2}} B(E) $$

Taking the duals of all sides gives the following.

$$ m^{-\frac{3}{2}} B(E)^{\circ} \subseteq P^{\circ} \subseteq (1-\varepsilon)^{-1} B(E)^{\circ} $$

Setting $Q=-m^{\frac{3}{2}}P^{\circ}$ and using $-B(E)^{\circ}=\widetilde{P}os_E - r$ we have the following inclusions.

$$ \widetilde{P}os_E - r \subseteq Q  \subseteq (1+ 2 \varepsilon) m^{\frac{3}{2}} (\widetilde{P}os_E - r)$$

Note that, by construction $P$ has $O(m)$ vertices and $Q$ has $O(m)$ facets.

\subsection{Improved Accuracy With More Facets}

In this section we will use a standard construction from multilinear algebra, namely the tensor power of a vector space. Tensor powers will help us to produce approximation polytopes with more facets and improved accuracy. Now, let $V$ be a vector space equipped with inner product $\langle \; , \;  \rangle$. For an integer $k \geq 1$, the $k$ th tensor power of $V$ is defined as follows. 

$$ V^{\otimes k} = \underbrace{V \otimes V \otimes V \otimes \ldots \otimes V}_{k \; \text{times}}$$

\noindent It is natural to consider $V^{ \otimes k}$ with the following inner product.

$$  \langle x_1\otimes x_2 \otimes \ldots \otimes x_k , y_1 \otimes y_2 \otimes \ldots \otimes y_k \rangle = \prod_{i}^k \langle x_i , y_i \rangle  $$

\noindent The symmetric part of $V^{\otimes k}$,  $Sym(V^{\otimes k})$   is the subspace spanned by tensors $x^{\otimes k}= x \otimes x \otimes \ldots \otimes x$.  We consider $Sym(E^{k})$, the symmetric part of the $k$ th tensor power of the polynomial space $E$. Recall that, we defined a subspace   $U(E) := \{ f \in E : \langle f , r \rangle = 0 \}$.  Note that $\dim \left( U(E) \right)= m-1$, and $\dim \left( Sym( U(E)^{\otimes k}) \right)= \binom{m+k-2}{k} $.  After this line, we write $U$ for $U(E)$ hoping that no confusion arises. 

We define the following object in $Sym(U^{\otimes k})$.

$$ B(E)^{ \otimes k} : = conv \{  \phi_E(v)^{\otimes k} : v \in S^{n-1} \} $$

\noindent We assume $B(E)^{ \otimes k}$ is in John's position and repeat the reasoning presented in the previous section. This proves existence of a set $S_k \subseteq S^{n-1}$ with $\abs{S_k} \leq \binom{m+k-2}{k} c \varepsilon^{-2}$, such that the following inequalities are satisfied for all $q^{\otimes k} \in U^{\otimes k}$.

$$  (1-\varepsilon) \max_{g^{\otimes k} \in B(E)^{ \otimes k}} \langle g^{\otimes k} , q^{\otimes k} \rangle \leq \max_{v \in S_k } \abs{ \langle f_v^{\otimes k} , q^{\otimes k} \rangle } \leq  \binom{m+k-2}{k}^{\frac{3}{2}} \max_{g^{\otimes k} \in B(E)^{ \otimes k}} \langle g^{\otimes k} , q^{\otimes k} \rangle  $$

We can also write this inequalities as follows.

$$  (1-\varepsilon) \max_{g^{\otimes k} \in B(E)^{ \otimes k}} \langle g , q \rangle^k \leq \max_{v \in S_k } \abs{ \langle f_v , q \rangle}^{k} \leq  \binom{m+k-2}{k}^{\frac{3}{2}} \max_{g^{\otimes k} \in B(E)^{ \otimes k}} \langle g , q \rangle^{k}  $$

\noindent We fix $\varepsilon=\frac{1}{2e}$, and we define $P_k := conv \{ f_v , -f_v : v \in S_k \}$. Then, we have the following inequalities.

$$ (1- \frac{1}{2e})^{\frac{1}{k}} \max_{g \in B(E) } \langle g , q \rangle \leq \max_{g \in P_k}  \langle g , q \rangle \leq  \binom{m+k-2}{k}^{\frac{3}{2k}} \max_{g \in B(E)} \langle g , q \rangle     $$

\noindent  The standard Stirling estimate for binomial coefficients gives $\binom{m+k-2}{k} \leq e^{m-2}(1+\frac{k}{m-2})^{m-2}$. If we set $k=s(m-2)$,  then we have

$$ \binom{m+k-2}{k}^{\frac{3}{2k}} \leq \left( e(1+\frac{k}{m-2}) \right)^{\frac{3(m-2)}{2k}} \leq \left( e(1+s) \right)^{\frac{3}{2s}} .$$ 

\noindent Note also that $(1- \frac{1}{2e})^{-1} \leq 1+\frac{1}{e}$. Since above inequalities on $\max_{P_k} \langle g , q \rangle$ hold for arbitrary $q \in U$ we have the following inclusions. 

$$  (1+\frac{1}{e})^{-\frac{1}{s(m-2)}} B(E) \subseteq  P_k \subseteq \left( e(1+s) \right)^{\frac{3}{2s}} B(E) $$

\noindent We set $Q_k= -  \left( e(1+s) \right)^{\frac{3}{2s}}  P_k^{\circ}$ and we assume $s \geq e$. Note that $(1+\frac{1}{e})^{\frac{1}{s(m-2)}} \leq (\frac{s+1}{e})^{\frac{1}{s}}$. Therefore,

 $$ (1+\frac{1}{e})^{\frac{1}{s(m-2)}} \left( e(1+s) \right)^{\frac{3}{2s}} \leq (s+1)^{\frac{3}{s}} .$$

\noindent Now using the relation $ -B(E)^{\circ} = \widetilde{P}os_E - r$, we conclude 

$$ \widetilde{P}os_E - r \subseteq Q_k \subseteq (1+\frac{k}{m-2})^{\frac{3(m-2)}{k}} (\widetilde{P}os_E - r ).$$

Let us summarize the result of this section in a theorem statement.

\begin{thm}
Let $E \subset P_{n,2d}$ be a linear subspace, let $r=(x_1^2+x_2^2+\ldots +x_n^2)^{d}$ and assume $r \in E$. Let $m=\dim(E)$, assume $m \geq 3$, and let $k$ be an integer with $ k \geq e(m-2)$. Then, there exists a polytope $Q_k$ with $O(k^{m-2})$ facets which satisfies 

$$ \widetilde{P}os_E - r \subseteq Q_k \subseteq (1+\frac{k}{m-2})^{\frac{3(m-2)}{k}} (\widetilde{P}os_E - r )  $$

\noindent In particular, $Q_n$ has $O(n^{m-2})$ many facets and it satisfies

$$ \widetilde{P}os_E - r \subseteq Q_n \subseteq (1+\frac{n}{m})^{\frac{3m}{n}} (\widetilde{P}os_E - r )  $$
\end{thm}

\begin{rem} \label{sasha}
Alexander Barvinok recently pointed out to me his article \cite{barvinok5}. The arguments in \cite{barvinok5} are well optimized for approximating arbitrary convex bodies with arbitrary polytopes. The optimization of the approximation in \cite{barvinok5} is done by using Chebyshev nodes. In our particular case, the order of approximation given by \cite{barvinok5} is not better than results of this article due to large coefficient of symmetry of the cone of nonnegative polynomials \cite{blekherman}. The main difference in our argument is that we do not aim to approximate with arbitrary polytopes, instead we approximate with polytopes created out of pointwise evaluations in $E$.  
\end{rem}
	
\section{Random Approximation Polytopes}
Deterministic construction of the approximation polytopes in previous section includes two main steps: computation of John's Ellipsoid and computation of an approximate decomposition of identity as stated in Theorem \ref{cufcuf}. The computation of John's Ellipsoid is known to be hard \cite{todd}. Therefore,  we will take a different approach in this section and construct approximation polytopes by random sampling. 

If one is equipped with the knowledge that $\norm{P_E(p_v)} \leq a \sqrt{\dim(E)}$ for some universal constant $a$, then one can bypass John's Ellipsoid computation and directly use the theorem of Rudelson  as explained in Remark \ref{rudelson}. However, as we have seen in the third section one does not always have such an upper bound. Hence, different tools are needed. 

We will use a tool coming from computational geometry, namely the $\varepsilon$-net theorem. It is hard to do any justice to beautiful mathematics behind the $\varepsilon$-net theorem in a very limited space. Therefore, we just state a special case of a version of the $\varepsilon$-net theorem due to Koml\'os, Pach and Woeginger \cite{komlos}, and refer the reader to \cite{matousek} and references therein. 
\begin{thm}{($\varepsilon$-net Theorem, special case for halfspaces)} Let $\mathcal{F}$ be a family of halfspaces in $\mathbb{R}^{m-1}$. Let $\mu$ be a probability measure on $\mathbb{R}^{m-1}$. Assume that for all halfspaces $H_{+} \in \mathcal{F}$, we have $ \mu(H_{+}) \geq \varepsilon $ where  $0 < \varepsilon \leq (3e)^{-2}$. 

Let $t=\left\lceil \frac{3m}{\varepsilon} \ln(\frac{1}{\varepsilon}) \right\rceil$, and let $x_1,x_2,\ldots, x_t$ be independent random variables distributed  with respect to $\mu$.  Then, the set $X=\{ x_1, x_2, \ldots, x_t \}$ has a non-empty intersection with all members of $\mathcal{F}$ with probability greater than $1- 4 (9 e^2 \varepsilon)^{m}$.
\end{thm}

Nasz\'odi's recent article \cite{naszodi} was a source of inspiration for the ideas developed in this section. One can obtain the above theorem as a special case of Lemma 3.2 in Naszodi's article by setting $C=3$, $\delta=4 (9 e^2 \varepsilon)^{m}$, and using the fact that a collection of halfspaces in $\mathbb{R}^{m-1}$ has VC-dimension at most $m$.

To use the $\varepsilon$-net theorem we need to specify a probability measure on our space. We will use a modified version of the  probability measure defined in the third section. Let us recall how the construction works. We have $E \subseteq P_{n,2d}$ with $\dim(E)=m$ and $r=(x_1^2+x_2^2+\ldots+x_n^2)^d \in E$. Then, we define a subspace $U$ as $U := \{ f \in E : \langle f ,r \rangle= 0\}$. We showed that for every $v \in S^{n-1}$, there exist a polynomial that we denoted by $\Pi_E(p_v-r) \in U$  such that for all $f \in U$,  we have $f(v)=\langle f , \Pi_E(p_v-r) \rangle$. Then, we defined a map $\phi_E$ as follows. 

$$ \phi_E:  S^{n-1} \rightarrow E $$
$$  \phi_E(v)= \Pi_E(p_v-r) $$

\noindent Out of this map, we created the following convex body.

$$ B(E) := conv \{ \phi_E(v) : v \in S^{n-1} \}$$

\noindent We showed that $B(E) \subset U$ is dual to a base of the cone of nonnegative polynomials, hence our objective is to approximate $B(E)$ with polytopes having as few vertices as possible. 

Now we define the probability measure that is going to be used in this section. Let $\lambda$ be the uniform measure on $\Delta^m$ where $\Delta^m$ is defined as follows. 

$$ \Delta^m := \{ x \in \mathbb{R}^m : x_i \geq 0 \; \text{and} \; \sum_{i} x_i = 1  \}$$

Now, let $T^{m} = S^{n-1} \times S^{n-1} \times \ldots \times S^{n-1}$ be the cartesian product of $S^{n-1}$ with itself $m$-times and let $\sigma_m = \sigma \times \sigma \times \ldots \times \sigma$ be the product measure on $T^{m}$ where $\sigma$ is the uniform measure on $S^{n-1}$. Let $\Psi$ be a map from $\Delta^{m} \times T^{m}$ to $U$ defined as follows.

$$ \Psi : \Delta^{m} \times T^{m} \rightarrow U$$

$$ \Psi( a_1, a_2, \ldots, a_m , v_1 , v_2, \ldots, v_m ) = \sum_{i=1}^m a_i \phi_E(v_i) $$

\noindent We define $\mu$ to be the pushforward measure of $\lambda \times \sigma_m$ under $\Psi$. $\mu$ is clearly a probability measure, and it is supported on $B(E)$ i.e., $\mu(B(E))=1$.

Now we need to pick a special family of halfspaces and show that their measure is bounded from below by a certain threshold. We consider the family of halfspaces defined by supporting hyperplanes of $\alpha B(E)$ for a fixed $0 < \alpha \leq 1$.

$$ \mathcal{F_{\alpha}}:= \{ H^{+} : H \; \text{is a supporting hyperplane of} \; \alpha B(E) \}$$

\noindent Suppose we have a set $V = \{v_1,v_2, \ldots, v_t \} \subseteq B(E) $ where for all halfspaces $H^{+} \in \mathcal{F_{\alpha}}$ there exist a $v_i \in H^{+}$. This would imply the following inclusions.

$$ \alpha B(E) \subseteq conv\{ f_1, f_2, \ldots , f_t\} \subseteq B(E) $$

\noindent Our goal in the rest of this section is to use this observation to construct random approximation polytopes to $B(E)$. We first need  some preparatory lemmas.  

\begin{lem} \label{meh}
Let $f \in P_{n,2d}$ be a polynomial with $\Int_{S^{n-1}} f(x) \; \sigma(x) = 0$ where $\sigma$ is the uniform measure on $S^{n-1}$. Then, we have

$$ \sigma( \{ x : f(x) > 0 \} ) \geq  \frac{\norm{f}_1^2}{4\norm{f}_2^2}$$
\end{lem}

\begin{proof}
We define a function $f_{+}$ as follows.
\begin{equation*}
  f_{+}(x)=\left\{
  \begin{array}{@{}ll@{}}
    f(x) & \text{if}\ f(x) > 0 \\
    0  & \text{otherwise}
  \end{array}\right.
\end{equation*}

\noindent Now consider $f_{+}(x)$ as a random variable where $x$ is distributed according to $\sigma$. Clearly $f_{+}(x) \geq 0$ for all $x \in S^{n-1}$, and we also have 

$$ \sigma( \{ x : f(x) > 0 \} ) = \sigma( \{ x : f_{+}(x) > 0 \} ) .$$

\noindent We apply Paley-Zygmund inequality to $f_{+}(x)$, which gives the following. 

$$ \sigma( \{ x : f_{+}(x) > 0 \} ) \geq \frac{(\mathbb{E} f_{+}(x))^{2}}{\mathbb{E} f_{+}(x)^2} $$

\noindent Observe that, $\Int_{S^{n-1}} f(x) \; \sigma(x) = 0$ implies $2 \Int_{S^{n-1}} f_{+}(x) \; \sigma(x)=\norm{f}_1$. Also we trivially have $\mathbb{E} f_{+}(x)^2 \leq \norm{f}_2^2$. 

\end{proof}
\noindent As a corollary of this lemma, we have the following lower bound on the measure of halfspaces from $\mathcal{F_{\alpha}}$. 

\begin{cor} \label{ok}
Let $H_{+} \in \mathcal{F_{\alpha}}$ be a halfspace defined by an $f \in U$ as follows.

$$H_{+} := \{ g \in U : \langle f , g \rangle \geq \alpha \max_{v \in S^{n-1}} f(v) \} $$

Then, we have

$$ \mu(H_{+}) \geq  \left( \frac{(1-\alpha) \norm{f}_1^2}{4\norm{f}_2^2} \right)^{m} $$.
\end{cor}

\begin{proof}
First, let us note that 
$$ \max_{g \in \alpha B(E)} \langle f , g \rangle=\alpha \max_{v \in S^{n-1}} f(v)$$ 

\noindent Therefore, all the supporting hyperplanes of $\alpha B(E)$ will be in the format considered in the corollary statement. Now let $H_{+}^{1}$ be the translate of $H_{+}$ supporting $B(E)$, and let $\phi_E(v_0) \in \partial B(E) \cap H_{+}^{1}$. Also let $H_{+}^{0}$ be the translate of $H_{+}$ passing through the origin. Then $\alpha \phi_E(v_0) + (1-\alpha) (B(E) \cap H_{+}^{0}) $ is included inside $B(E) \cap H_{+}$ due to convexity. 

Now let $A_f := \{ v \in S^{n-1} : f(v) \geq 0 \}$, and let $A_f^m = A_f \times A_f \times \ldots \times A_f$ be the cartesian product of $A_f$ with itself $m$-times. Then, the image of $(1-\alpha) \Delta^m \times A_f^m$ under the map $\Psi$ is included in $(1-\alpha) (B(E) \cap H_{+}^{0})$. By the definition of the measure $\mu$, we then have the following lower bound.

$$ \mu(H_{+}) \geq \mu \left( \alpha \phi_E(v_0) +(1-\alpha) (B(E) \cap H_{+}^{0}) \right) \geq (1-\alpha)^{m-1} \sigma(A_f)^m \geq (1-\alpha)^{m} \sigma(A_f)^m $$

\noindent Using Lemma \ref{meh} completes the proof. 

$$ \mu(H_{+})  \geq \left( \frac{(1-\alpha)\norm{f}_1^2}{4\norm{f}_2^2} \right)^{m}  $$

\end{proof}
Now we are ready to state and prove the main result of this section. 

\begin{thm} \label{epsilon}
Let $E \subset P_{n,2d}$ be a subspace with $r \in E$ and $\dim(E)=m$. Let $B(E)$ and the measure $\mu$ be as defined above. We set $M(E)=\max_{f \in E} \frac{\norm{f}_2^2}{\norm{f}_1^2}$. Now let $t=\left\lceil 27 e^2 m^2 (\frac{ M(E)}{1-\alpha})^m  \ln(\frac{M(E)}{1-\alpha}) \right\rceil$, and let $x_1,x_2, \ldots,x_t$ be independent random vectors in $E$ distributed according to $\mu$. Then,

$$ \alpha B(E) \subseteq conv \{ x_1, x_2, \ldots, x_t \} \subseteq B(E) $$

\noindent with probability at least $1-4(\frac{1-\alpha}{M(E)})^{m^2}$.
\end{thm}
\begin{proof}
We will use the $\varepsilon$-net theorem with the measure $\mu$ and the family of halfspaces $\mathcal{F}_{\alpha}$. From Corollary \ref{ok},  we have that for all $H_{+} \in \mathcal{F}_{\alpha}$, $\mu(H_{+}) \geq (\frac{1-\alpha}{4M(E)})^m$. We set $\varepsilon=\frac{(1-\alpha)^m}{9e^2M(E)^m}$ for which we clearly have $9e^2\varepsilon \leq 1$, and also $\mu(H_{+}) \geq \varepsilon$ for all $H_{+} \in \mathcal{F}_{\alpha}$. Now, let $t=\left\lceil \frac{3m}{\varepsilon} \ln(\frac{1}{\varepsilon}) \right\rceil$ and let $x_1,x_2,\ldots,x_t$ be independent random vectors distributed according to $\mu$. Then, the $\varepsilon$-net theorem yields that $X= \{ x_1, x_2, \ldots, x_t \}$ is a transversal of $\mathcal{F}_{\alpha}$ with probability at least $1-4(9e^2\varepsilon)^{m}$.
\end{proof}

Below is the restatement of the theorem with dual convex bodies.

\begin{cor}

Let $E \subset P_{n,2d}$ be a subspace with $r \in E$ and $\dim(E)=m$. Let $\mu$ be the measure as defined above. We set $M(E)=\max_{f \in E} \frac{\norm{f}_2^2}{\norm{f}_1^2}$. For a given $\alpha$ with $0 < \alpha \leq 1$, we set $t=O( m^2 (\frac{ M(E)}{1-\alpha})^m \ln(\frac{M(E)}{1-\alpha}))$. Let $x_1,x_2,\ldots,x_t$ be independent random vectors distributed with $\mu$. We define the polytope $K_{\alpha}= \{ y \in U : \langle y , x_i \rangle \geq -1 \; \text{for all} \; 1 \leq i \leq t \}$. Then $K_{\alpha}$ satisfies

$$ \widetilde{P}os_E - r \subseteq K_{\alpha} \subseteq \frac{1}{\alpha}(\widetilde{P}os_E - r) $$

\noindent with probability at least $1-4(\frac{1-\alpha}{M(E)})^{m^2}$.
\end{cor}

We do not know any way to bound $M(E)$ from above using the low-dimensionality of $E$. What we can do is to bound $M(P_{n,2d})$ from above in terms of $d$ only.

\begin{lem} \label{yep}
The following holds for all $f \in P_{n,2d}$.
$$ \frac{\norm{f}_2^2}{\norm{f}_1^2} \leq  2^{2d}  $$
\end{lem}

\begin{proof}

$$ \Int_{\mathbb{R}^n} \abs{f(x)}\; \gamma_n(x) = \frac{\abs{S^{n-1}}}{(2\pi)^{\frac{n}{2}}} \Int_{0}^{\infty} r^{n+2d-1} e^{\frac{-r^2}{2}} \; dr \Int_{S^{n-1}} \abs{f(x)} \; \sigma(x) $$

$$ \Int_{\mathbb{R}^n} \abs{f(x)}\; \gamma_n(x)=  \frac{\abs{S^{n-1}}}{(2\pi)^{\frac{n}{2}}} 2^{\frac{n}{2}+d-1} \Gamma(\frac{n}{2}+d) \Int_{S^{n-1}} \abs{f(x)} \; \sigma(x) $$

Following the same pattern we also derive the following.

$$ \Int_{\mathbb{R}^n} f(x)^2 \; \gamma_n(x) = \frac{\abs{S^{n-1}}}{(2\pi)^{\frac{n}{2}}} 2^{\frac{n}{2}+2d-1} \Gamma(\frac{n}{2}+2d) \Int_{S^{n-1}} f(x)^2 \; \sigma(x)  $$

Also note that, $\frac{\abs{S^{n-1}}}{(2\pi)^{\frac{n}{2}}} 2^{\frac{n}{2}} = \frac{2}{\Gamma(\frac{n}{2})}$. This gives us the following ratio.

$$ \frac{\Int_{\mathbb{R}^n} f(x)^2 \; \gamma_n(x)}{(\Int_{\mathbb{R}^n} \abs{f(x)}\; \gamma_n(x))^2}= \frac{\norm{f}_2^2}{\norm{f}_1^2} \frac{\Gamma(\frac{n}{2}) \Gamma(\frac{n}{2}+2d)}{\Gamma(\frac{n}{2}+d)^2} $$

The reverse H\"older inequalities for polynomials with normal random variables \cite{bobkov,pro}  tell us that the left hand side of the inequality is less than $2^{2d}$. Therefore, we have
$$ \frac{\norm{f}_2^2}{\norm{f}_1^2} \leq  2^{2d} \frac{\Gamma(\frac{n}{2}+d)^2}{\Gamma(\frac{n}{2}) \Gamma(\frac{n}{2}+2d)} \leq 2^{2d} $$

\end{proof} 

\section{Acknowledgements}
I had the idea of Theorem \ref{inapprox} after having a short discussion with Alexander Barvinok  on December 2016 at ICERM, Brown University. I also had a chance to meet with him at  Physikzentrum Bad Honnef on late November 2017. While discussing about this note, Prof. Barvinok pointed out his paper \cite{barvinok5} where he approximates arbitrary convex bodies using spectral sparsification. I included a remark at the end of fifth section section about the comparison of this work with results in \cite{barvinok5}. I would like thank Alexander Barvinok for enjoyable discussions, and to ICERM and Physikzentrum Bad Honnef for their hospitality. I also would like to present my thanks to Cynthia Vinzant and Seth Sullivant for discussions at Raleigh on Fall 2016, and to anonymous referees whose remarks helped to clarify the presentation.

\end{document}